\documentclass[final, nomarks]{dmtcs-episciences}
\usepackage{amsmath}
\usepackage[utf8]{inputenc}
\usepackage{subfigure}

\usepackage[round]{natbib}
\usepackage{tikz}
\newcommand\polya{{P\' olya}}
\newcommand\vecX{{\bf X}}
\newcommand\Prob{{\mathbb P}}
\newcommand\E{{\mathbb E}}
\newtheorem{cor}{Corollary}[section]
\newtheorem{theorem}{Theorem}[section]
\newtheorem{rem}{Remark}[section]

\author[Srinivasan Balaji and Hosam Mahmoud]{Srinivasan Balaji\affiliationmark{1}
  \and
Hosam Mahmoud\affiliationmark{1}}

\title[Urn-driven random walks]{Urn-driven random walks}
\affiliation{
Department of Statistics, The George Washington University, Washington,  U.S.A. 
}
\begin{document}

\publicationdata{vol. 28:2}{2026}{24}{10.46298/dmtcs.15546}{2025-04-23; 2025-04-23; 2026-01-12}{2026-03-25}
\maketitle

\vskip 0.25cm

\begin{abstract}
The simple symmetric random walk is known to be recurrent in one and two dimensions, and becomes transient in three or higher dimensions. We compare the simple symmetric random walk to walks driven by certain  \polya\ urns. We show that, in contrast, if the probabilities of the random walk are instead driven by a \polya-Eggenberger urn, the states are recurrent only in one dimension.

Further consideration of exchangeability reveals that the walk is null recurrent.
As soon as the underlying Markov chain of \polya\ walk
gets in two dimensions or higher, there is a positive 
probability that the walker gets lost in the space, and the probability of her recurrence is less than 1.

On the other hand, a walk driven by Friedman urn behaves like the simple symmetric random walk, being recurrent in one and two dimensions and transient in higher dimensions.
As Friedman urn scheme is not exchangeable, it is considerably harder to determine the nature of
the recurrence in one and  two dimensions. Empirical evidence through simulation suggests that
in one dimension Friedman walk is positive recurrent. 
\end{abstract}
\section{Introduction}
The simple symmetric random walk in one dimension (the horizontal axis) 
progresses as follows. A particle starts out at the origin. At each unit of discrete time, the particle moves \emph{forward} one unit of distance in the positive direction, or moves \emph{backward}   
one unit of distance in the negative direction with equal probability.
If after a number of moves the particle is at position $i \in \mathbb Z$, we say 
the random walk (Markov chain) is in state $i$. We denote the state after~$n$ moves by $X_n$. 
It is well known that the states of the chain are \emph{recurrent}, that is, the particle's  
probability of returning to a specific state infinitely often is~1.   
A necessary and sufficient condition for this is 
$$\sum_{n=0}^\infty \Prob(X_n = 0) = \infty.$$
If a state is not recurrent, it is called \emph{transient}.
Now, a random walk in~$d$ dimensions 
arises by having a particle 
in $d$ dimensions changing one coordinate by a unit (positive or negative) at each point in discrete time, all moves being equally likely.
For instance, in three dimensions, the particle chooses one of six possible displacements, which are namely, east, west, north, south, up or down, and each 
move has probability $\frac 1 6$.
   It is well known that the 1-dimensional and 2-dimensional random walks are recurrent,
but the $d$-dimensional random walks, for $d\ge 3$, are all transient. 
This is a classic result first published by \polya\  (see~\cite{Polya}).

There is a variety of new ways to generate a random walk. For instance, the elephant random walk is in the vogue. This walk was introduced in~\cite{Trimper}. In recent years, numerous studies
have been published on it; see for example~\cite{Bercu,Bertoin,Shuo}.

Our purpose is to compare the recurrence of the states, if the random walk follows a different probability model driven by a symmetric urn scheme, such as moving according to 
a \polya-Eggenberger\ urn scheme or a Friedman urn scheme.  
\section{\polya\ random walk}
First, a word about \polya\ urns is in order to set  up the random walk.

A \polya-Eggenberger urn scheme is a system comprised of an urn that contains
balls of up to two colors, which we call white and blue, and rules for drawing
and reactions.
We execute a sequence of draws in discrete time units.
At each unit of time, a ball is drawn from the urn at random. The ball is returned 
to the urn, together with a ball of the same color.

Let $W_n$ and $B_n$ be respectively the number of white draws and blue draws
among $n$ draws. Hence, 
after $n$ draws, we have $W_n + B_n = n$.
We suppose the urn starts with $w>0$ white balls and $b>0$ 
blue balls. Therefore, after~$n$ draws, the urn has $W_n+w$ white balls and
$B_n+b$ blue balls. 
Note that we excluded the cases $w=0$ or $b=0$, as they lead to degeneracy. 

We associate with the \polya-Eggenberger urn scheme a linear random walk in one dimension, 
which we take to be the horizontal axis.
At the beginning of time, a particle is placed at the origin; i.e., the walk starts at the origin. After each drawing, the particle moves one unit of distance in the positive direction, if the ball drawn is white,
or one unit of distance in the negative direction, if the ball drawn is blue.
Thus, right after the  $n$th draw, the probability of moving forward is $(W_n+w)/  (n+ w + b)$,
and the probability of moving backward is $(B_n+b)/  (n+ w + b)$.

Let $X_n$ be the position of \polya\ random walk
after $n$ draws. \polya\ walk is in a state specified by the position of the moving 
particle and the number of draws from each color. After $r$ white draws and
$s$ blue draws, we have $X_n =x$, for some $x\in \mathbb Z$.  Consider this as the state $(x,r,s)$.
At this point, the number of white balls in the urn is $w+r$ and the number of blue balls 
in the urn is $b+s$.
Transition into the state $(x+1, r+1, s)$ occurs with probability $(w+r)/(n+w+b)$ and
transition into the state $(x-1,r,s+1)$ occurs with probability $(b+s)/(n+w+b)$. 

Note that in the state
$(x,r,s)$ the numbers $r$ and $s$ are determined from the relations 
$$r+s=n, \qquad x = r-s. $$
In other words, we have $r = r(x) = (n+x)/2$ and $s=s(x)=(n-x)/2$. So,
all the information about the state is latent in $x$. Therefore, we can 
consider \polya\ walk as a
Markov chain on $X_n$. With the
appropriate number of draws, we can move from any state to any other. So, this Markov chain is irreducible.

Let us dwell on calculating $\Prob(X_n = 0)$, the probability 
of returning to the origin after $n$ moves. 
Clearly, this probability is 0,
after an odd number of moves. The particle is at the origin, if it made an equal
number of forward and backward moves.
Since $B_{2n} = 2n - W_{2n}$,
we have
$$\Prob(X_{2n} = 0)  = \Prob(W_{2n} = n, B_{2n} = n).$$

The main results need an asymptotic estimate for the exact order of 
$\Prob(X_{2n} = 0) $. As we shall see in the proofs, this probability turns out
to be of the exact order $1/n$.
\begin{theorem}
\label{Theo:walk}
In the one-dimensional \polya\ random walk state 0 is null recurrent.
\end{theorem}
\begin{proof}
The joint
distribution of the number of white and blue draws is the well-known \polya\ distribution; see~\cite{Kotz} or~\cite{Mahmoud}. 
\polya\ distribution is conveniently described in terms 
of Pochhammer's symbol for the rising factorial, defined by
$$ \langle x \rangle_s = x(x + 1)(x + 2)\cdots(x + s - 1), $$
for any $x \in \mathbb{R}$, and any integer $s \ge 0$, with the interpretation that $\langle x \rangle_0 = 1$. 

According to \polya\ distribution, we have
\begin{align*}
\Prob(X_{2n} = 0) &= \frac {\langle w \rangle_n \langle b \rangle_n}{
   \langle w+b\rangle_{2n}} \, {2n \choose n}\\ 
&= K \, \frac {(n+w-1)!\,  (n+b-1)!}{
   (2n+w+b-1)!} \times \frac { (2n)!}   
   {(n!)^2},
\end{align*}     
where $K := (w+b-1)! / ((w-1)!\, (b-1)!)$.
This formula is a special case of \polya\ distribution  (also called the beta–binomial distribution), which first appeared in~\cite{Eggenberger}. Specialized to  the case of~$n$ white draws and $n$ blue draws, the formula follows from the exchangeability
of the indicators of the draws. So, one can just lump all $n$ indicators of white draws at the beginning and count 
the number of sequences with $n$ white draws.

To find bounds on the probabilities, 
we employ the well-known Stirling approximation to the factorial: 
\begin{equation}
\Bigl(\frac n e\Bigr)^n \sqrt {2\pi n} \  \ \le \  \ n! \ \ \le \ \ \Bigl(\frac n e\Bigr)^n
 \sqrt {2\pi n}\, e^{\frac 1 {12n}}.
\label{Eq:Stirling}
\end{equation}

We first prove that $\Prob(X_{2n} = 0)$ is
$\Omega(1/n)$.\footnote{One says that a function $g(n)$ is $\Omega(h(n))$, if there exist a natural number $n_0$ and a positive constant $C$, such that $C |h(n)| \le |g(n)|$, for all $n\ge n_0$.} 
Using the bounds in~(\ref{Eq:Stirling}), we write
 \begin{align*}
\Prob(X_{2n} = 0) &\ge 
   K \, \frac {{(\frac{n+w-1} e)^{n+w-1}}  \sqrt {2\pi (n+w-1)} 
               \, {(\frac{n+b-1} e)^{n+b-1}}  \sqrt {2\pi (n+b-1)} }{
   {(\frac{2n+w+b-1} e)^{2n+w+b-1}}  \sqrt {2\pi (2n+w+b-1)}
                  \, e^{\frac 1 {12(2n+w+b-1)}}} \\
              &\qquad\qquad  {}\times \frac {(\frac{2n} e)^{2n}  \sqrt {4\pi n}}   
   {((\frac{n} e)^{n}  \sqrt {2\pi n}\, e^{\frac 1 {12n}})^2}.
\end{align*} 
Using the fact that $w\ge 1$ and $b \ge 1$ gives a simpler inequality, for $n\ge 1$, which is namely
\begin{align*}
\Prob(X_{2n} = 0) &\ge  Ke   \, \frac {(n^{n+w-1}  \sqrt  n\,)
                (n^{n+b-1}  \sqrt  n\, )}{
   {(2n+w+b-1)^{2n+w+b-1}\sqrt{2n+w+b-1}\, e^{\frac 1 {12}}}} \\
             &\qquad\qquad \qquad  {} \times \frac {(2n)^{2n}  \sqrt {2n}}   
   {(n^n  \sqrt{n}\, e^{\frac 1 {12}})^2} \\
   &\ge  \frac {K e^{\frac 3 4}} {2^{w+b-1}}\times \frac {1
                }{
   {n \, (1 +\frac{w+b-1}{2n})^{2n+w+b-1}}} \times \frac {\sqrt{2n}} 
      {\sqrt{2n + w +b -1}}.
\end{align*}
For $x\ge 1$ and every $\alpha > 0$, 
the function $\sqrt {x/ (x+\alpha)}$ is bounded below by $\sqrt {1/ (1+\alpha)}$, and the function
$(1 + \alpha/x)^{x+\alpha}$ is bounded above by $e^\alpha (1+\alpha)^\alpha$. We thus have the lower bound
$$
\Prob(X_{2n} = 0) 
   \ge  \frac {K  e^{\frac 3 4}} {2^{w+b-1}}\times \frac 1
	                {n e^{w+b-1} (w+b)^{w+b-1}} \times \frac 1 
      {\sqrt{1 + \frac 1 2 (w +b -1)}} =: \frac {C_1} n ;
$$
the coefficient $C_1 > 0$ collects all the constants.

Now, we have
$$\sum_{n=0}^\infty \Prob(X_{2n} = 0)  \ge \sum_{n=1}^\infty \Prob(X_{2n} = 0) 
   \ge   \sum_{n=1}^\infty \frac {C_1} n = \infty,$$
proving that state 0 is recurrent.

We next determine the kind of recurrence in \polya\ random walk.
 Starting at the origin, consider the variable $H_0$, the first hitting (return) 
 time to state 0. Represent the sequence of draws (moves) by a string of white ($W$) draws and
 blue ($B$) draws. For example $WWBWWWBBBB$ represents a sequence of draws from the urn corresponding to two white draws (two moves forward) followed by a blue draw (a backward move), etc. This sequence of moves returns the particle to the origin in ten transitions.
 
 The number of drawing sequences starting with a white draw and returning the particle to the origin in~$2n$ moves
 can be mapped onto left (W) and right (B) balanced parentheses, with the number of left parentheses exceeding the number of right parentheses till position $2n$ with  every prefix of length $k\le 2n-1$ having more left parentheses than right. 
 As an example, the given instance is mapped into (()((()))). 
 
 Let $Q_{2n}$ be the number of such formal expressions. 
 For instance, with $n=3$, we have the correspondence
\begin{equation}
 WW BWBB \implies (()()), \qquad WWWBBB\implies ((())). \label{Eq:Cat}
 \end{equation}
 We can count $Q_{2n}$ by outflanking any balanced expression of~$2n-2$ parentheses between an opening left and a closing right parenthesis. The first left parenthesis guarantees that, by the balancing of the outflanked expression, there will be more left than right parentheses, till position $2n-1$ and the right parenthesis at position $2n$ balances the expression. The number of well balanced parentheses of length $2n-2$ is well known to be the Catalan number ${2n-1 \choose n-1}/ n$; see~\cite[Section 1.3, pp.~12--14]{Stanley}. 
 
 In total we have  ${2n-2 \choose n-1}/ n$ drawing sequences returning the particle to the origin in $2n$ moves starting with a white draw. 
 In the running example, $Q_6 = {4\choose 2}/3 = 2$. There are two possible drawing sequences from the urn starting with a white draw, then returning the particle to the origin in six draws with the number of white draws exceeding the number of blue draws until time 5--two expressions of balanced parentheses with every prefix of length $k\le 5$ having more left parentheses than right---as shown in Display~(\ref{Eq:Cat}).
 
 By the exchangeability of the \polya\ urn scheme, all these sequences are equiprobable, with the probability of each being that of a sequence of $n$ $W$'s in the prefix and $n$ $B$'s in the suffix. Thus, we have
 $$\Prob(H_0 = 2n) = 2\, \frac {\langle w\rangle_n\langle b\rangle_n}
     {\langle w+b\rangle_{2n}}{2n-2 \choose n-2} \frac 1 n;$$
     note the doubling to take into account drawing sequences starting with $B$.  
Using the well-known asymptotics for the Catalan numbers~\cite[Section 2.3, pp.~35–37]{Stanley} and developments via
the Stirling approximation of the factorials (as was done in the first part of the
proof), we find
$$\Prob(H_0 = 2n) \sim 2 \sqrt \pi \, n \times \frac {4^{n-1}} {\sqrt{\pi}\, n^{3/2}} 
        = 2\,  \frac {4^{n-1}} { n^{1/2}}. $$
 Consequently, the average return time is
 $$\E[H_0] = \sum_{n=0}^\infty n\, \Prob(H_0 = n) 
      = \sum_{n=0}^\infty 2n\,\Prob(H_0 = 2n) \sim  \sum_{n=0}^\infty 
                     n\Big(\frac {4^{n}}{\sqrt n} \Big) = \infty. $$       
    Hence, state 0 in the \polya\ random walk is null recurrent.
\end{proof}

As all the states communicate, we have the following corollary.
\begin{cor}
\label{Cor:walk}
The one-dimensional \polya\ random walk is null recurrent.
\end{cor}
\section{\polya\ random walk in higher dimensions}
Let us transcend the linear motion and consider random walks in $d\ge 2$ dimensions. 
Suppose there are $d\ge 2$ independent 
\polya\ random walks. 
Denote the position of the $i$th walk by $X_n^{(i)}$, for $i=1, \ldots, d$. At time~$n$, the walk is at the point ${\bf X}_n = (X_n^{(1)}, \ldots, X_n^{(d)})$ of the $d$-dimensional space. Again, the particle cannot be at the origin unless it made
an even number of moves in each dimension.  
By the independence in each dimension,
the probability of being at the $d$-dimensional  origin ${\bf 0} = (\underbrace{0, 0, \ldots, 0}_{d\ \mbox{times}})$ is
$$  \Prob({\bf X}_{2n} = {\bf 0}) = \prod_{i=1}^d \Prob(X_{2n}^{(i)} = 0) . $$

\begin{rem} 
\label{Rem:rotation}
The $d$-dimensional random \polya\ walk does not place the moving
particle at the same positions as the simple symmetric $d$-dimensional random walk.
For instance, after one move, in the 2-dimensional \polya\ walk the particle can be at one of the four positions
$(\pm1,\pm 1)$, whereas the simple symmetric random walk places it at one of the four positions
$(1,0)$, $(0,1)$,$(-1,0)$ or $(0,-1)$. The two walks, however are equivalent via rotations and change of scale.
\end{rem}
Figure~\ref{Fig:walk} illustrates Remark~\ref{Rem:rotation}.
We can describe the path of \polya\ walk in two dimensions in terms of ordered pairs
 $(x,y) \in\{W,B\} \times \{W,B\}$
specifying the colors drawn from the first and second urns, respectively.

The simple symmetric random walk is generated by flipping two independent fair coins, with Heads($H$)
and tails ($T$) outcome. Going forward (backward)
in each dimension is achieved on a Heads (Tails) outcome.   
We can describe the path of the simple symmetric walk in terms of ordered pairs $(x,y)\in\{H,T\} \times \{H,T\}$
specifying the outcomes of the first and second coins, respectively.

The path of \polya\ walk in Figure~\ref{Fig:walk} (shown to the left) corresponds 
to the sequence of draws
$$(W,W), (W,W), (B,W), (B,B),(B,W),(B,W),(B,W),(W,W),$$
and the path of the simple symmetric random walk is generated by the flips
$$(H,H), (H,H), (T,H), (T,T),(T,H),(T,H),(T,H),(H,H).$$
The choice of the nearest neighbor is 
$$(H,H) = \mbox{east},\quad (T,H) = \mbox{north}, \quad(T,T) = \mbox{west},\quad (H,T) = \mbox{south}.$$

The probability of the shown path for \polya\ walk depends on the initial conditions in the two urns.
For instance, if the first urn starts out with one white ball and one blue ball and the second urn
starts out with two white balls and one blue ball, the path has probability 
\begin{align*}
&\Big(\frac 1 2 \times \frac 2 3 \times \frac 1 4 \times \frac 2 5\times \frac  3 6 \times \frac 4 7 \times \frac 5 8 \times \frac 3 9\Big)  
\Big(\frac  2 3 \times \frac  3 4\times \frac  4 5\times \frac 1 6\times \frac  5 7 \times \frac 6 8 \times \frac 7 9 \times \frac 8 {10}\Big)\\
&\qquad =\frac 1 {504} \times \frac1 {45}\\
&\qquad =\frac 1 {22680} \\
&\qquad \approx  0.0000441.
\end{align*} 
With the pair of dice being fair (unbiased), the same path in the two-dimensional 
simple symmetric 
random walk has probability $1/4^8 =1/655361 \approx 0.0000152588$.

Observe that the path in the left picture of Figure~\ref{Fig:walk} is  obtained from the one in the right by a $45$ degrees rotation in the anticlockwise direction, and a scaling by $\sqrt 2$.

\bigskip
\begin{figure}[bht]
\begin{center}
\begin{tikzpicture}[scale=0.6, >=stealth]
  \draw[->] (-5,0) -- (4,0) node[right] {$x$};  
  \draw[->] (0,-0.5) -- (0,7) node[above] {$y$};  
\draw[->] (5.5,0) -- (10,0) node[right] {$x$};
  \draw[->] (6,-0.5) -- (6,7) node[above] {$y$};  


\node at (0, 0) {$\bullet$};
\node at (1, 1) {$\bullet$};
\node at (2, 2) {$\bullet$};
\node at (1,3) {$\bullet$};
\node at (0, 2) {$\bullet$};
\node at (-1, 3) {$\bullet$};
\node at (-2, 4) {$\bullet$};
\node at (-3, 5) {$\bullet$};
\node at (-2, 6) {$\bullet$};
\coordinate (A) at (0,0);
\coordinate (B) at (1,1);
\coordinate (C) at (2,2);
\coordinate (D) at (1,3);
\coordinate (E) at (0,2);
\coordinate (F) at (-1,3);
\coordinate (G) at (-2,4);
\coordinate (H) at (-3,5);
\coordinate (I) at (-2,6);
\draw [thick] (A)--(B)--(C)--(D)--(E)--(F)--(G)--(H)--(I);

\node at (6, 0) {$\bullet$};
\node at (7, 0) {$\bullet$};
\node at (8, 0) {$\bullet$};
\node at (8,1) {$\bullet$};
\node at (7, 1) {$\bullet$};
\node at (7, 2) {$\bullet$};
\node at (7,3) {$\bullet$};
\node at (7, 4) {$\bullet$};
\node at (8, 4) {$\bullet$};
\coordinate (A) at (6,0);
\coordinate (B) at (7,0);
\coordinate (C) at (8,0);
\coordinate (D) at (8,1);
\coordinate (E) at (7,1);
\coordinate (F) at (7,2);
\coordinate (G) at (7,3);
\coordinate (H) at (7,4);
\coordinate (I) at (8,4);
\draw [thick] (A)--(B)--(C)--(D)--(E)--(F)--(G)--(H)--(I);

\end{tikzpicture}
\end{center}
  \caption{A \polya\ random walk (left); the corresponding simple symmetric 
  random walk (right).}
  \label{Fig:walk}
\end{figure}
\begin{theorem}
In the d-dimensional \polya\ random walk, with $d\ge 2$,
state~0 is transient.
\end{theorem}
\begin{proof}
To establish the result, we need an upper bound 
on $\Prob(X_{2n}^{(i)} = 0)$. The proof is very similar to what we encountered
in the proof of Theorem~\ref{Theo:walk} as an argument for the lower bound;
we only highlight the starting point and
minor differences. Using the bounds in~(\ref{Eq:Stirling}), we have
 \begin{align*}
\Prob(X_{2n} = 0) &\le 
   K \, \frac {{(\frac{n+w-1} e)^{n+w-1}}  \sqrt {2\pi (n+w-1)} 
               \, {(\frac{n+b-1} e)^{n+b-1}}  \sqrt {2\pi (n+b-1)} }{
   {(\frac{2n+w+b-1} e)^{2n+w+b-1}}  \sqrt {2\pi (2n+w+b-1)}} \\
              &\qquad\qquad  {}\times \frac {(\frac{2n} e)^{2n}  \sqrt {4\pi n}}   
   {((\frac{n} e)^{n}  \sqrt {2\pi n}} \times e^{\frac 1 {12(n+w-1)}} \times e^{\frac 1 {12(n+b-1)}}
   \times e^{\frac 1 {12\times 2n}}.
\end{align*} 
We can replace the three exponential functions involving $\frac 1 {12}$ by the uniform bound
$e^{\frac 1 {12} + \frac 1 {12} + \frac 1 {12\times 2}} = e^{\frac 5 {24}}$. From this point on,
the proof is almost identical to that of the upper bound in the proof of Theorm~\ref{Theo:walk}, reaching the conclusion that
$$\Prob(X_{2n} = 0) \le  \frac {C_2} n,$$
for some constant $C_2>0$, for all $n\ge 1$. 
It follows that
$$ \sum_{n=0}^\infty \Prob({\bf X}_{2n} = {\bf 0}) = \prod_{i=1}^d \Prob(X_{2n}^{(i)} = 0) 
           <    1 + C_2^d \sum_{n=1}^\infty
                 \frac 1 {n^d} < \infty.      $$  
state 0 is transient, for $d \ge 2$.     
\end{proof}   

As all the states communicate, we have the following corollary. 
\begin{cor}
The $d$-dimensional \polya\ random walk, with $d\ge 2$,
is transient.
\end{cor}   
\section{Friedman random walk}
We first say a word about Friedman urns in order to set up a corresponding random walk.

A Friedman urn scheme is a system comprised of an urn that contains white and blue balls,
and evolves according to rules of opposite reinforcement (unlike \polya-Eggenberger self-reinforcement strategy). We execute a sequence of draws in discrete time units. At each epoch, a ball is drawn from the urn at random, and then returned to the urn together with a ball of the opposite color.

It is the nature of Friedman's scheme to revert to a state of equilibrium, where in the long run we have an equal average number of balls of each color. This should be intuitively transparent---when one color dominates, the probability is high to pick balls of that color, with the net effect of compensating the deficit in the other color. This mean-reverting behavior induces a high probability of a small difference between the number of white balls and that of blue balls. A random walk driven by Friedman urn is therefore likely to stay near the origin. Such a walk stays recurrent in more dimensions than \polya\ random walk. 

A random walk driven by Friedman urn proceeds in the following way. We start with one white 
ball in the urn and a particle at the origin. At each epoch of discrete time, a ball is drawn from the urn and the rules of Friedman urn scheme are followed. The particle moves forward, if the 
color of the ball drawn is white, or moves backward,  if the 
color of the ball drawn is blue.

By time $n$, let $W_n$ ($B_n$) be the number of  times a white (blue) ball is picked,
and let $\widetilde W_n$ ($\widetilde B_n$) be the number of white (blue) balls in the urn.

In~\cite{Dumas}, the authors use an analytic method, to determine the probability distribution of 
$\widetilde W_n$, and for an urn starting with only one white ball find
$$\Prob(\widetilde W_n = k) = \frac {\big\langle {n \atop k} \big\rangle}{n!}, \qquad \mbox{for \ } k = 1, \ldots, n,$$
where 
$\big\langle {n \atop k} \big\rangle$ are the Eularian numbers of the first kind, that count the number of permutations of the set $\{1, 2, \ldots, n\}$ with exactly $k$ ascents 
(see~\cite[Section 6.2]{Graham}). For a broader discussion
of the analytic method for treating urns see~\cite{Gabarro}.    

Note the connections $\widetilde B_n = W_n$ and $\widetilde W_n + \widetilde B_n  = n+1$ , to 
have 
$$\Prob(W_n = k) = \Prob(\widetilde B_n = k) = \Prob(n+1 -  \widetilde W_n= k).$$  
As discussed, the random walk returns to the origin in $2n$ moves, if $W_{2n} = n$. 
So, we have
$$\Prob(X_{2n}=0) = \Prob(W_{2n} = n) =  \frac {\big\langle {2n \atop n+1} \big\rangle}{(2n)!}.$$             
The asymptotics of  $\langle {n \atop k} \big\rangle$, for fixed $k$, as $n\to\infty$, had been 
known for a long time. More recently,~\cite{Keller} completed the spectrum
by finding the asymptotics for all $k$. Actually,~\cite{Keller} develops the asymptotics of $A(n,k)$, which are the
number of breaking points between the ascents. In other words, we have $A(n,k) = \big\langle {n \atop k+1}\big\rangle$. For our purpose, we are seeking the asymptotics of 
$\langle {2n \atop n+1} \big\rangle = A(2n,n)$. By the symmetry of the Eulerian numbers, we have $A(2n,n) = A(2n, n-1) $. For even $n$, it is reported in~\cite{Keller} that
$$A\Big(n,  \frac 1 2 n -1\Big)  \sim \sqrt {12}\, \Big(\frac n e\Big)^n,
              \qquad \mbox{as\ } n \to \infty.$$ 
By an application of Stirling approximation to the factorial in our context, we find  
$$\Prob(X_{2n}=0) \sim \frac {\sqrt{12}\,\big( {\frac{ 2n}  e}\big)^{2n}}{\big({\frac {2n}  e}\big)^{2n}
\sqrt{2\pi (2n)}} = \sqrt{\frac 3 {\pi n}}.$$
\subsection{Friedman random walk in higher dimensions}
We can now consider a $d$--dimensional Friedman random walk, wherein the movement
in each dimension is driven by an independent Friedman urn.
At time~$n$, Friedman random walk is at the point ${\bf X}_n = (X_n^{(1)}, \ldots, X_n^{(d)})$ of the $d$-dimensional space. Again, the particle cannot be at the origin unless it made
an even number of moves in each dimension.  
By the independence in each dimension,
the probability of being at the $d$-dimensional  origin ${\bf 0} = (\underbrace{0, 0, \ldots, 0}_{d\ \mbox{times}})$ is
$$  \Prob({\bf X}_{2n} = {\bf 0}) = \prod_{i=1}^d \Prob(X_{2n}^{(i)} = 0) . $$
\begin{theorem}
\label{Thm:Friedman}
State {\bf 0} in the $d$-dimensional Friedman random walk, is recurrent for $d=1$ and $d=2$. With $d\ge 3$,
state 0 is transient.
\end{theorem}
\begin{proof}
In one dimension we have
$$\sum_{n=1}^\infty \Prob(X_n = 0)  = \sum_{n=1}^\infty \Prob(X_{2n} = 0)
   \sim \sum_{n=1}^\infty\sqrt{\frac 3 {\pi n}}= \infty,$$
and state 0 is recurrent.

In two dimension we have
$$\sum_{n=1}^\infty \Prob(\vecX_n = {\bf 0})  = \sum_{n=1}^\infty \Prob(\vecX_{2n} 
               = {\bf 0})
   \sim \sum_{n=1}^\infty \bigg(\sqrt{\frac 3 {\pi n}}\,\bigg)^2 = \infty,$$
and state $\bf 0$ is recurrent.
 
In $d \ge 3$ dimensions, we have
$$\sum_{n=1}^\infty \Prob(\vecX_n = {\bf 0})  = \sum_{n=1}^\infty \Prob(\vecX_{2n} = {\bf 0})
   \sim \sum_{n=1}^\infty \bigg(\sqrt{\frac 3 {\pi n}}\,\bigg)^d < \infty,$$
and the origin is transient.
\end{proof}  

As all the states communicate, we have the following corollary.
\begin{cor}
The $d$-dimensional Friedman random walk, is recurrent in one and two dimensions. In 
$d\ge 3$ dimensions, the walk 
is transient.
\end{cor}    
\section{Empirical evidence of positive recurrence}
The exchangeability in \polya\ urn scheme enables us to determine the nature
of \polya\ walk through the distribution of the number of white draws among~$n$ drawings;
the walk turns out to be null recurrent.
The lack of such exchangeability in Friedman urn scheme proves to be a major 
challenge to go through an argument similar
to the one used in the \polya\ random walk to determine the nature of the recurrence
in Theorem~\ref{Thm:Friedman}. 

Our strategy shifts to generating a conjecture through simulation of the 1-dimensional Friedman walk. The simulation is carried
out by the algorithm below. The algorithm receives $replica$, a predetermined number and simulates  Friedman's urn that many times. The algorithm stores the number of white draws in $W$ and the number of blue draws in $B$. The simulation starts after one draw. At the beginning of each replication,
the variables are initialized as
$$W\gets1; \qquad B \gets 0.$$ 
According to this initialization, there are more white draws than blue at the beginning of any replication.

The algorithm keeps going at simulating draws from a Friedman's urn till $W=B$, at which point
the corresponding walk returns to the origin. 
The algorithm employs the primitive $\bf random$, which generates a random number
uniformly distributed over the interval (0,1)---most modern programing languages and computing 
packages provide such a random number generator. If the generated number is less than 
$(B+1)/(B+W +1)$ (the proportion of white balls), we take the draw to be white, otherwise it is blue. 
The updates after each drawing are straightforward: For a white draw,
the algorithm increases the number of white draws $W$ by 1; otherwise, the algorithm increases the number of blue draws $B$ by 1. At the end of the $r$th replication, the total number of draws till the first 
return to the origin $W+B$ is stored at the $r$th position of the array $H$. 

After all the desired replications are carried out, the array $H$ contains statistical data on
$H_0$, the first hitting time of 0. The algorithm in pseudocode is given next.   

 \bigskip\noindent
\phantom{X} \hskip 0.5 cm {\bf for} $r$ {\bf from} 1 {\bf to} replica {\bf  do} \\
 \phantom{X} \hskip 1 cm {\bf  while} ($W \not = B$) {\bf  do} \\
\phantom{X} \hskip 1.5   cm      $U \gets {\bf  random}$ \\
\phantom{X} \hskip 1.5    cm     {\bf  if} $U < \frac {B+1} {W +B +1}$  {\bf then}   $W\gets W+1$ \\
   \phantom{X} \hskip 1.5  cm    {\bf  else}       $B\gets B+ 1$\\ 
   \phantom{X} \hskip 1 cm $H_r  \gets W+B$
 
 \bigskip
We can take the statistical average of the entries in the array
$H$ as an estimator $\overline{H_0}$ of $H_0$. 
We ran the algorithm with $replica = 10000$, and in each run we got a value of $\overline{H_0}$ ranging between 8 and 12.
Backed with this empirical evidence, it is our conjecture that the average first hitting time is 
finite and that  Friedman's random walk in one dimension is positive recurrent. 
A proof of this conjecture
remains a future research project.

A minor adaptation of the algorithm is needed for the 2-dimensional Friedman walk. We constructed 
such a simulation and  the results are not conclusive to determine the type of recurrences.
\acknowledgements
The authors thank two anonymous referees for constructive comments that 
improved the exposition.

\nocite{*}
\bibliographystyle{abbrvnat}
\bibliography{refs.txt}

\begin{thebibliography}{13}
\providecommand{\natexlab}[1]{#1}
\providecommand{\url}[1]{\texttt{#1}}
\expandafter\ifx\csname urlstyle\endcsname\relax
  \providecommand{\doi}[1]{doi: #1}\else
  \providecommand{\doi}{doi: \begingroup \urlstyle{rm}\Url}\fi

\bibitem[Bercu and Laulin(2019)]{Bercu}
B.~Bercu and L.~Laulin.
\newblock On the multi-dimensional elephant random walk.
\newblock \emph{Journal of Statistical Physics}, 175:\penalty0 1146--1163,
  2019.

\bibitem[Bertoin(2022)]{Bertoin}
J.~Bertoin.
\newblock Counting the zeros of an elephant random walk, 2022.
\newblock arXiv:2105.09569.

\bibitem[Eggenberger and P{\'o}lya(1923)]{Eggenberger}
F.~Eggenberger and G.~P{\'o}lya.
\newblock {\"U}ber die statistik verketteter vorg{\"a}nge.
\newblock \emph{Zeitschrift f{\"u}r Angewandte Mathematik und Mechanik},
  3:\penalty0 279--289, 1923.

\bibitem[Flajolet et~al.(2005)Flajolet, Gabarr{\'o}, and Pekari]{Gabarro}
P.~Flajolet, J.~Gabarr{\'o}, and H.~Pekari.
\newblock Analytic urns.
\newblock \emph{The Annals of Probability}, 33:\penalty0 1200--1233, 2005.

\bibitem[Flajolet et~al.(2006)Flajolet, Dumas, and Puyhaubert]{Dumas}
P.~Flajolet, P.~Dumas, and V.~Puyhaubert.
\newblock Some exactly solvable models of urn process theory.
\newblock In P.~Chassaing, editor, \emph{Discrete Mathematics and Computer
  Science Proceedings}, volume~AG, pages 59--118, 2006.

\bibitem[Giladi and Keller(1994)]{Keller}
E.~Giladi and J.~Keller.
\newblock Eulerian number asymptotics.
\newblock \emph{Proceedings of the Royal Society A}, 445:\penalty0 291--303,
  1994.

\bibitem[Graham et~al.(1994)Graham, Knuth, and Patashnik]{Graham}
R.~Graham, D.~Knuth, and O.~Patashnik.
\newblock \emph{Concrete Mathematics: A Foundation for Computer Science}.
\newblock Addison-Wesley, Reading, Massachusetts, 1994.

\bibitem[Johnson and Kotz(1977)]{Kotz}
N.~Johnson and S.~Kotz.
\newblock \emph{Urn Models and Their Application}.
\newblock Wiley, New York, 1977.

\bibitem[Mahmoud(2008)]{Mahmoud}
H.~Mahmoud.
\newblock \emph{P{\'o}lya Urn Models}.
\newblock Chapman \& Hall, Boca Raton, Florida, 2008.

\bibitem[P{\'o}lya(1921)]{Polya}
G.~P{\'o}lya.
\newblock {\"U}ber eine aufgabe der wahrscheinlichkeits\-theorie betreffend die
  irrfahrt im stra{\ss}ennetz.
\newblock \emph{Mathematische Annalen}, 84:\penalty0 149--160, 1921.

\bibitem[Qin(2025)]{Shuo}
S.~Qin.
\newblock Recurrence and transience of multidimensional elephant random walks.
\newblock \emph{Annals of Probability}, 53:\penalty0 1049--1078, 2025.

\bibitem[Sch{\"u}tz and Trimper(2004)]{Trimper}
G.~Sch{\"u}tz and S.~Trimper.
\newblock Elephants can always remember: Exact long-range memory effects in a
  non-markovian random walk.
\newblock \emph{Physical Review E}, 70:\penalty0 045101, 2004.

\bibitem[Stanley(2015)]{Stanley}
R.~Stanley.
\newblock \emph{Catalan Numbers}.
\newblock Cambridge University Press, Cambridge, UK, 2015.

\end{thebibliography}

\end{document}